\documentclass[11pt]{amsart}

\usepackage{amsmath,amssymb,amsthm}
\usepackage{mathtools}
\usepackage{enumitem}

\theoremstyle{plain}
\newtheorem{theorem}{Theorem}[section]

\newtheorem{corollary}[theorem]{Corollary}

\theoremstyle{definition}

\theoremstyle{remark}
\newtheorem{remark}[theorem]{Remark}

\newcommand{\Fq}{\mathbb{F}_q}
\newcommand{\Fqstar}{\mathbb{F}_q^*}

\title[permutation trinomials and complete permutation polynomials]{On permutation trinomials and complete permutation polynomials
  via fiber criteria over finite fields }

\author{Chahrazade Bouyacoub, Asmae El-Baz and Omar Kihel}


\date{\today}

\begin{document}

\begin{abstract}
We give new, short  proofs of recent permutation polynomial
results of Bousalmi, Bayad, and Derbal by reducing the verification to
explicit computations on a three-element multiplicative subgroup via
Zieve's fiber criterion.  Building on this approach, we develop a general
framework---combining Zieve's theorem with the AGW criterion---for
constructing complete permutation polynomials over finite fields through
a fiber decomposition over the cube roots of unity.  A scalar
specialization of the criterion yields families that are easy to produce
and verify.  We illustrate the construction with concrete examples and
show through counterexamples that the underlying arithmetic conditions
are sharp.
\end{abstract}

\maketitle

\section{Introduction}\label{sec:intro}

Permutation polynomials over finite fields---polynomials
$f\in\Fq[X]$ that induce bijections on~$\Fq$---have attracted
sustained attention since the foundational work of Hermite and Dickson
in the nineteenth century.  Their study lies at the crossroads of number
theory, combinatorics, and algebra, and has gained further momentum
through applications to coding theory, cryptography, and the design of
pseudorandom sequences.  A natural strengthening is the notion of a
\emph{complete permutation polynomial} (CPP): a polynomial~$f$ such
that both~$f$ and~$f+X$ permute~$\Fq$.  Complete permutation
polynomials impose a substantially more rigid condition and are
correspondingly harder to construct.

Among the most effective modern tools for establishing the permutation
property are two reduction principles.  The first, due to Zieve~\cite{Zieve2009},
applies to polynomials of the form
$f(X)=X^r h\!\left(X^{(q-1)/d}\right)$ and reduces the permutation
question over~$\Fq$ to two conditions: a coprimality requirement on~$r$
and a permutation test on the small multiplicative subgroup~$\mu_d$.
The second, the AGW criterion of Akbary, Ghioca, and Wang~\cite{AGW2011},
provides a general fiber-wise framework: it decomposes a bijection on a
finite set into a permutation of fibers together with bijectivity on
each individual fiber.  Both tools have become standard in the
construction of new permutation families.  In this direction, Ayad, Belghaba, and
Kihel~\cite{AyadBelghabaKihel2014,AyadBelghabaKihel2015} studied
permutation binomials $ax^n+x^m$ over finite fields, establishing
non-existence results under sharp arithmetic conditions on the
parameter $d=\gcd(n-m,q-1)$ and bounding the size of the base
field in terms of~$d$ alone.

In 2021, Bousalmi, Bayad, and Derbal~\cite{BBD2021} presented explicit
families of permutation trinomials of the form
$X^r(X^{2(q-1)/3}+\alpha X^{(q-1)/3}+\beta)$ over~$\Fq$ with
$q\equiv 1\pmod{3}$, parametrized either by a cube root of unity or by
a field element satisfying a simple arithmetic constraint (Theorems~3
and~4 of~\cite{BBD2021}).  Their proofs rely on direct evaluation and case
analysis.  Our first contribution is to give short and transparent
alternative proofs of these results by a systematic application of
Zieve's criterion with $d=3$: the verification reduces in each case to
an explicit computation on the three-element group~$\mu_3$.

Our second, and main, contribution concerns complete permutation
polynomials.  We establish a general criterion, based on the AGW
framework, for polynomials of the form
$f(X)=X^r c(X^{(q-1)/3})$ to be CPPs of~$\Fq$ when
$q\equiv 1\pmod{9}$.  The criterion decomposes into four verifiable
conditions---a coprimality and non-vanishing hypothesis, fiber
injectivity, a well-definedness condition on the induced map, and a
permutation test on~$\mu_3$---and unifies the treatment of a range of
trinomial shapes.  We then isolate a \emph{scalar fiber specialization}
that arises when $r\equiv 1\pmod{(q-1)/3}$: in this regime, the fiber
maps become scalar multiplications and the four conditions collapse to
simple non-vanishing and permutation checks.  The resulting family
criterion is easy to apply and produces explicit CPP families.  Finally,
we show through concrete counterexamples that the congruence
$q\equiv 1\pmod{9}$ is essential for our construction, in the sense that
natural parameter choices fail when only $q\equiv 1\pmod{3}$ is assumed.

\section{Preliminaries}\label{sec:prelim}

\subsection{Notation and standing assumptions}
Throughout, $\Fq$ denotes a finite field of order~$q$ with
$q\equiv 1\pmod{3}$, and we set
\[
  s \;:=\; \frac{q-1}{3}\,.
\]
The multiplicative group~$\Fqstar$ contains a unique subgroup of
order~$3$,
\[
  \mu_3 \;:=\; \{u\in\Fqstar : u^3=1\}
        \;=\; \{1,\,\omega,\,\omega^2\},
\]
where $\omega$ is a fixed primitive cube root of unity, so that
$\omega^2+\omega+1=0$.  The map $\pi\colon\Fqstar\to\mu_3$ defined by
$\pi(x)=x^s$ is a surjective group homomorphism with kernel
\[
  K \;:=\; \ker\pi \;=\; \{z\in\Fqstar : z^s=1\},
\]
a subgroup of order~$s$.  For each $u\in\mu_3$, the fiber
$\pi^{-1}(u)=tK$, where $t$ is any element of~$\Fqstar$ with $t^s=u$,
is a coset of~$K$ of cardinality~$s$.  The three fibers
$\pi^{-1}(1)$, $\pi^{-1}(\omega)$, $\pi^{-1}(\omega^2)$
partition~$\Fqstar$.

The following result, due to Zieve~\cite{Zieve2009}, reduces the question
of whether certain structured polynomials permute~$\Fq$ to a permutation
test on the small group~$\mu_d$.

\begin{theorem}[Zieve {\cite{Zieve2009}}]\label{thm:zieve}
Let $d$ and $r$ be positive integers with $d\mid q-1$, and let
$h\in\Fq[X]$.  Then
\[
  f(X) \;=\; X^r\, h\!\left(X^{(q-1)/d}\right)
\]
is a permutation polynomial of\/~$\Fq$ if and only if
\begin{enumerate}[label=\textup{(\arabic*)},leftmargin=2em]
\item $\gcd\!\left(r,\,\dfrac{q-1}{d}\right)=1$, and
\item the map $u\mapsto u^r\, h(u)^{(q-1)/d}$ is a permutation
  of\/~$\mu_d$.
\end{enumerate}
\end{theorem}

In this paper we apply Theorem~\ref{thm:zieve} exclusively with $d=3$,
so that condition~(1) becomes $\gcd(r,s)=1$ and condition~(2) is a
permutation test on the three-element group~$\mu_3$.

\begin{theorem}[AGW criterion {\cite[Lemma~1.2]{AGW2011}}]
\label{thm:agw}
Let $A$, $S$, $\bar S$ be finite sets with $|S|=|\bar S|$, and let
$g\colon A\to A$, $\bar g\colon S\to\bar S$,
$\lambda\colon A\to S$, $\bar\lambda\colon A\to\bar S$ be maps
satisfying
\[
  \bar\lambda\circ g \;=\; \bar g\circ\lambda.
\]
If both $\lambda$ and $\bar\lambda$ are surjective, then the following
are equivalent:
\begin{enumerate}[label=\textup{(\roman*)},leftmargin=2em]
\item $g$ is a bijection on~$A$;
\item $\bar g$ is a bijection from $S$ to $\bar S$, and $g$ is injective
  on $\lambda^{-1}(s)$ for each $s\in S$.
\end{enumerate}
\end{theorem}

In this paper we apply the AGW criterion with $A=\Fqstar$,
$S=\bar S=\mu_3$, and $\lambda=\bar\lambda=\pi\colon x\mapsto x^s$.

\section{Short proofs of the Bousalmi--Bayad--Derbal theorems}
\label{sec:bbd}

In this section we give new proofs of Theorems~3 and~4 of~\cite{BBD2021}
by direct application of Zieve's Theorem~\ref{thm:zieve} with $d=3$.
In each case, the verification reduces to an explicit computation
on~$\mu_3=\{1,\omega,\omega^2\}$.

\begin{theorem}[Theorem~3 of {\cite{BBD2021}}]\label{thm:bbd3}
Let $q\equiv 1\pmod{3}$ and let $r$ be a positive integer with
$\gcd(r,q-1)=1$.
\begin{enumerate}[label=\textup{(\arabic*)},leftmargin=2em]
\item For every $\delta\in\mu_3$ such that
  $(3\delta^2+1)^s=1$ in~$\Fq$, the polynomial
  \[
    f(X) \;=\; X^r\!\left(X^{2s}+\delta\, X^s+\delta^2+1\right)
  \]
  is a permutation polynomial of\/~$\Fq$.
\item For any $\gamma\in\Fq\setminus\{1,-2\}$ such that
  $\left(\frac{\gamma+2}{\gamma-1}\right)^s=1$ in~$\Fq$, the polynomial
  \[
    f(X) \;=\; X^r\!\left(X^{2s}+X^s+\gamma\right)
  \]
  is a permutation polynomial of\/~$\Fq$.
\end{enumerate}
\end{theorem}

\begin{proof}
Since $\gcd(r,q-1)=1$ implies $\gcd(r,s)=1$, condition~(1) of
Theorem~\ref{thm:zieve} is satisfied in both parts.  It remains to
verify that $u\mapsto u^r h(u)^s$ permutes~$\mu_3$, where $h$ is
the appropriate quadratic polynomial.

\medskip
\noindent\textit{Part~(1).}\;
Define $h_\delta(X)=X^2+\delta X+\delta^2+1$.  We evaluate~$h_\delta$
on each element of~$\mu_3$, using the relation
$1+\omega+\omega^2=0$ throughout.

If $\delta=1$, then
\[
  h_1(1)=4, \qquad
  h_1(\omega)=\omega^2+\omega+2=1, \qquad
  h_1(\omega^2)=\omega+\omega^2+2=1.
\]
If $\delta=\omega$, then
\[
  h_\omega(1)=1+\omega+\omega^2+1=1, \qquad
  h_\omega(\omega)=\omega^2+\omega^2+\omega^2+1=3\omega^2+1, \qquad
  h_\omega(\omega^2)=1.
\]
If $\delta=\omega^2$, then
\[
  h_{\omega^2}(1)=1, \qquad
  h_{\omega^2}(\omega)=1, \qquad
  h_{\omega^2}(\omega^2)=3\omega+1.
\]
In each case, exactly one of the three values equals $3\delta^2+1$ and
the remaining two equal~$1$.  By hypothesis $(3\delta^2+1)^s=1$, so
$h_\delta(u)^s=1$ for every $u\in\mu_3$.  The induced map therefore
satisfies $u\mapsto u^r h_\delta(u)^s=u^r$ for all $u\in\mu_3$.  Since
$\gcd(r,3)=1$ (a consequence of $\gcd(r,q-1)=1$ and $3\mid q-1$), the
map $u\mapsto u^r$ is a permutation of the cyclic group~$\mu_3$.  By
Theorem~\ref{thm:zieve}, $f$ is a permutation polynomial of~$\Fq$.

\medskip
\noindent\textit{Part~(2).}\;
Define $h_\gamma(X)=X^2+X+\gamma$.  Evaluating on~$\mu_3$:
\[
  h_\gamma(1)=2+\gamma, \qquad
  h_\gamma(\omega)=\omega^2+\omega+\gamma=\gamma-1, \qquad
  h_\gamma(\omega^2)=\gamma-1.
\]
The three images of $u\mapsto u^r h_\gamma(u)^s$ on~$\mu_3$ are
therefore
\[
  (2+\gamma)^s, \qquad
  \omega^r(\gamma-1)^s, \qquad
  \omega^{2r}(\gamma-1)^s.
\]
We verify that these are pairwise distinct.  The hypothesis
$\left(\frac{\gamma+2}{\gamma-1}\right)^s=1$ gives
$(2+\gamma)^s=(\gamma-1)^s$.  The three pairwise ratios are
\[
  \frac{(2+\gamma)^s}{\omega^r(\gamma-1)^s}
  = \frac{1}{\omega^r}\,, \qquad
  \frac{\omega^r(\gamma-1)^s}{\omega^{2r}(\gamma-1)^s}
  = \frac{1}{\omega^r}\,, \qquad
  \frac{(2+\gamma)^s}{\omega^{2r}(\gamma-1)^s}
  = \frac{1}{\omega^{2r}}\,.
\]
Since $\gcd(r,3)=1$, neither $\omega^r$ nor $\omega^{2r}$ equals~$1$,
so all three ratios differ from~$1$.  The three images are pairwise
distinct, hence $u\mapsto u^r h_\gamma(u)^s$ permutes~$\mu_3$.  By
Theorem~\ref{thm:zieve}, $f$ is a permutation polynomial of~$\Fq$.
\end{proof}

\begin{theorem}[Theorem~4 of {\cite{BBD2021}}]\label{thm:bbd4}
Let $q\equiv 1\pmod{3}$, let $p=\mathrm{char}(\Fq)$, and let $r$ be a
positive integer with $\gcd(r,q-1)=1$.
\begin{enumerate}[label=\textup{(\arabic*)},leftmargin=2em]
\item If $q\equiv 1\pmod{6}$, then
  $f(X)=X^r\!\left(X^{2s}+X^s+\frac{p-1}{2}\right)$
  is a~PP of\/~$\Fq$.
\item If $p\equiv 1\pmod{3}$ and $q=p^{3k}$ with $k\ge1$, then
  $f(X)=X^r(X^{2s}+X^s+2)$ is a~PP of\/~$\Fq$.
\item If $p\ge5$, $p\equiv -1\pmod{3}$, and $q=p^{2k}$ with $k\ge1$,
  then $f(X)=X^r(X^{2s}+X^s+2)$ is a~PP of\/~$\Fq$.
\end{enumerate}
\end{theorem}

\begin{proof}
Each item is a special case of Theorem~\ref{thm:bbd3}(2) with a
specific value of~$\gamma$: it suffices to verify that
$\left(\frac{\gamma+2}{\gamma-1}\right)^s=1$ under the given
hypotheses, after which the distinctness argument from the proof of
Theorem~\ref{thm:bbd3}(2) applies verbatim.

\medskip
\noindent\textit{Item~(1): $\gamma=(p-1)/2$, $q\equiv 1\pmod{6}$.}\;
We compute
\[
  \frac{\gamma+2}{\gamma-1}
  = \frac{p+3}{p-3}
  \equiv \frac{3}{-3} = -1 \pmod{p}.
\]
Since $q\equiv 1\pmod{6}$, the integer $s=(q-1)/3$ is even, hence
$(-1)^s=1$ in~$\Fq$.

\medskip
\noindent\textit{Item~(2): $\gamma=2$, $p\equiv 1\pmod{3}$,
$q=p^{3k}$.}\;
Here $(\gamma+2)/(\gamma-1)=4$.  We show that $p-1\mid s$.  Writing
\[
  \frac{q-1}{p-1} = 1+p+\cdots+p^{3k-1},
\]
this sum has $3k$ terms, each congruent to~$1\pmod{3}$ (since
$p\equiv 1\pmod{3}$), so the sum is divisible by~$3$.  Hence
$s=(q-1)/3$ is divisible by $p-1$.  Since $4\in\mathbb{F}_p^*$ and
$\mathbb{F}_p^*$ has order $p-1$, we obtain $4^s=1$.

\medskip
\noindent\textit{Item~(3): $\gamma=2$, $p\ge5$,
$p\equiv -1\pmod{3}$, $q=p^{2k}$.}\;
Again $(\gamma+2)/(\gamma-1)=4$ and we show $p-1\mid s$.  The sum
\[
  \frac{q-1}{p-1} = 1+p+\cdots+p^{2k-1}
\]
has $2k$ terms with alternating residues $1,-1,1,-1,\ldots$ modulo~$3$
(since $p\equiv -1\pmod{3}$).  The terms pair off, each pair
contributing $1+(-1)=0\pmod{3}$, so the sum is divisible by~$3$.  Hence
$p-1\mid s$ and $4^s=1$.
\end{proof}

\section{A general criterion for index-3 complete permutation
  polynomials}\label{sec:cpp}

Having established in Section~\ref{sec:bbd} that Zieve's criterion
provides an efficient test for the permutation property of
$f(X)=X^r c(X^s)$, we now turn to the more demanding question: when is
$f$ a \emph{complete} permutation polynomial, that is, when does
$F(X)=f(X)+X$ also permute~$\Fq$?  The structure of~$F$ precludes a
direct application of Zieve's theorem, since the additive term~$X$
breaks the multiplicative shape $X^r h(X^s)$.  Instead, we appeal to the
AGW criterion (Theorem~\ref{thm:agw}), which allows us to decompose the
bijectivity of~$F$ on~$\Fqstar$ into a permutation test on~$\mu_3$
together with injectivity on each fiber of the projection
$\pi\colon x\mapsto x^s$.

\begin{theorem}
\label{thm:cpp-general}
Let $q\equiv 1\pmod{3}$, let $s=(q-1)/3$, and let $r$ be a positive
integer.  Let $c\colon\mu_3\to\Fqstar$ be a function, extended
to~$\Fqstar$ via $c(x^s)$, and define
\[
  f(X) = X^r\, c(X^s), \qquad F(X) = f(X)+X.
\]
Assume the following conditions hold:
\begin{enumerate}[label=\textup{(G\arabic*)},leftmargin=3em]
\item\label{G1} $\gcd(r,s)=1$, $\gcd(r,3)=1$, and $c(u)^s=1$ for all
  $u\in\mu_3$.
\item\label{G2} For each $u\in\mu_3$ and each $t\in\Fqstar$ with
  $t^s=u$, defining $\beta_u:=c(u)\,t^{r-1}$, the map
  \[
    \varphi_u\colon K\longrightarrow\Fqstar, \qquad
    z\longmapsto z\!\left(1+\beta_u\, z^{r-1}\right),
  \]
  is injective.
\item\label{G3} For each $u\in\mu_3$, the value
  \[
    v(u) \;:=\; \left(1+\beta_u\, z^{r-1}\right)^s
  \]
  is independent of the choice of $z\in K$, and lies in~$\mu_3$.
\item\label{G4} The map $\bar\psi\colon\mu_3\to\mu_3$ defined by
  $\bar\psi(u)=u\cdot v(u)$ is a permutation of~$\mu_3$.
\end{enumerate}
Then $f$ is a complete permutation
polynomial of\/~$\Fq$.
\end{theorem}

\begin{proof}
By~\ref{G1}, $\gcd(r,s)=1$ and $c(u)^s=1$ for all $u\in\mu_3$, so
the map $u\mapsto u^r c(u)^s=u^r$ is a permutation of~$\mu_3$ (since
$\gcd(r,3)=1$).  Zieve's Theorem~\ref{thm:zieve} with $d=3$ then
implies that $f$ is a PP of~$\Fq$.

We show that $F$ permutes~$\Fqstar$ by applying the AGW criterion
(Theorem~\ref{thm:agw}) with the following data:
\[
  A=\Fqstar, \quad S=\bar S=\mu_3, \quad
  \lambda=\bar\lambda=\pi, \quad g=F, \quad
  \bar g=\bar\psi.
\]
Since $\pi\colon\Fqstar\to\mu_3$ is a surjective group homomorphism,
the surjectivity hypothesis on $\lambda$ and $\bar\lambda$ is
satisfied.  It remains to establish the commutativity
$\pi\circ F=\bar\psi\circ\pi$ and the two requirements of
condition~(ii).

Fix $u\in\mu_3$ and choose $t\in\Fqstar$ with $t^s=u$.  Every element
of the fiber $\pi^{-1}(u)=tK$ can be written uniquely as $x=tz$
with $z\in K$, and a direct computation gives
\[
  F(tz) = (tz)^r c\!\left((tz)^s\right)+tz
        = t^r z^r c(u)+tz
        = tz\bigl(1+\beta_u\, z^{r-1}\bigr),
\]
where $\beta_u:=c(u)\,t^{r-1}\ne 0$.  Applying~$\pi$ to both sides:
\begin{equation}\label{eq:commut}
  \pi\bigl(F(tz)\bigr) = (tz)^s\bigl(1+\beta_u\, z^{r-1}\bigr)^s
                       = u\cdot\bigl(1+\beta_u\, z^{r-1}\bigr)^s.
\end{equation}
By~\ref{G3}, the factor $(1+\beta_u\,z^{r-1})^s=v(u)$ is independent
of $z\in K$, so~\eqref{eq:commut} gives
\[
  \pi\bigl(F(x)\bigr) = u\cdot v(u) = \bar\psi(u)
                       = \bar\psi\bigl(\pi(x)\bigr)
\]
for every $x\in\pi^{-1}(u)$.  Since $u\in\mu_3$ was arbitrary, the
commutativity $\pi\circ F=\bar\psi\circ\pi$ holds on all of~$\Fqstar$.

By~\ref{G4}, $\bar\psi\colon\mu_3\to\mu_3$ is a bijection.  For
the injectivity on fibers, note that $F(tz)=t\cdot\varphi_u(z)$
where $\varphi_u(z)=z(1+\beta_u\,z^{r-1})$.  Since $t\ne 0$, the
equality $F(tz_1)=F(tz_2)$ forces $\varphi_u(z_1)=\varphi_u(z_2)$,
hence $z_1=z_2$ by~\ref{G2}.  Thus $F$ is injective on~$\pi^{-1}(u)$
for each $u\in\mu_3$.

Both requirements of condition~(ii) in Theorem~\ref{thm:agw} are
satisfied, so $F$ is a bijection on~$\Fqstar$.  Since $F(0)=0$, the
map~$F$ permutes~$\Fq$.
\end{proof}

\begin{remark}\label{rem:independence}
Conditions~\ref{G2} and~\ref{G3} involve a choice of
representative~$t$ for the fiber $\pi^{-1}(u)$, but neither condition
depends on this choice.  Indeed, replacing~$t$ by $t'=tz_0$ with
$z_0\in K$ replaces $\beta_u$ by $\beta_u'=\beta_u\, z_0^{r-1}$ and
the map $\varphi_u(z)$ by $\varphi_u(z_0^{-1}z')$, which is injective
if and only if~$\varphi_u$ is.  Similarly, as $z$ ranges over~$K$, the
products $z_0^{r-1}z^{r-1}$ range over the same set as $z^{r-1}$, so
the independence in~\ref{G3} is preserved.
\end{remark}

\begin{remark}\label{rem:scalar-transition}
When $r\equiv 1\pmod{s}$, write $r-1=ks$.  Then for every $z\in K$ we
have $z^{r-1}=(z^s)^k=1$, so condition~\ref{G3} holds automatically
with $v(u)=(1+\beta_u)^s$, and~\ref{G2} reduces to
$1+\beta_u\ne 0$ (a simple non-vanishing).  This specialization is the
subject of the next section.
\end{remark}
\section{The scalar fiber specialization}\label{sec:scalar}

When $r\equiv 1\pmod{s}$, the general criterion of
Theorem~\ref{thm:cpp-general} simplifies considerably: the fiber maps
become scalar multiplications, and the four conditions
\ref{G1}--\ref{G4} reduce to elementary non-vanishing and permutation
checks.  This specialization is the most practical form of the
criterion, and is the one we apply in the examples of
Section~\ref{sec:examples}.

\begin{theorem}[Scalar fiber CPP theorem]\label{thm:scalar}
Let $q\equiv 1\pmod{9}$, let $s=(q-1)/3$, and write $r=1+ks$ for some
positive integer~$k$.  Let $c\colon\mu_3\to\Fqstar$ be a function with
$c(u)^s=1$ for all $u\in\mu_3$, and define
\[
  f(X) = X^r\, c(X^s), \qquad F(X) = f(X)+X.
\]
For each $u\in\mu_3$, set $\tau(u):=1+c(u)\,u^k$ and
$v(u):=\tau(u)^s$.  Assume:
\begin{enumerate}[label=\textup{(H\arabic*)},leftmargin=3em]
\item\label{H1} $\tau(u)\ne 0$ for all $u\in\mu_3$.
\item\label{H2} The map $\bar\psi\colon\mu_3\to\mu_3$ defined by
  $\bar\psi(u)=u\cdot v(u)$ is a permutation of~$\mu_3$.
\end{enumerate}
Then $f$ is a complete permutation
polynomial of\/~$\Fq$.
\end{theorem}

\begin{proof}
Since $q\equiv 1\pmod{9}$, we have $3\mid s$, hence
$r=1+ks\equiv 1\pmod{3}$, which gives $\gcd(r,3)=1$.  Moreover,
$\gcd(r,s)=\gcd(1+ks,s)=1$.  Together with the hypothesis $c(u)^s=1$
for all $u\in\mu_3$, we see that the map $u\mapsto u^r c(u)^s=u^r$ is
the identity on~$\mu_3$ (since $r\equiv 1\pmod{3}$), hence a
permutation.  Zieve's Theorem~\ref{thm:zieve} then implies that $f$ is
a PP of~$\Fq$.

We verify the conditions of Theorem~\ref{thm:cpp-general} for~$F$.
Fix $u\in\mu_3$ and $t\in\Fqstar$ with $t^s=u$.  For any $z\in K$,
the identity $z^{r-1}=z^{ks}=(z^s)^k=1$ simplifies the fiber
expression to
\[
  F(tz) = tz\bigl(1+c(u)\,t^{r-1}\bigr).
\]
Since $t^{r-1}=t^{ks}=(t^s)^k=u^k$, the scalar $\tau(u)=1+c(u)\,u^k$
is independent of both $z$ and the choice of representative~$t$, and is
nonzero by~\ref{H1}.  Thus
\[
  F(tz) = \tau(u)\cdot tz
\]
for every $z\in K$: the map~$F$ acts on the fiber $\pi^{-1}(u)$ as
multiplication by the nonzero constant~$\tau(u)$, which is manifestly
injective.  This verifies~\ref{G2}.

The identity $z^{r-1}=1$ on~$K$ makes~\ref{G3} automatic: for every
$z\in K$,
\[
  \bigl(1+\beta_u\,z^{r-1}\bigr)^s = \bigl(1+c(u)\,u^k\bigr)^s
  = \tau(u)^s = v(u),
\]
which is independent of~$z$ and lies in~$\mu_3$.

Finally, condition~\ref{G4} asks that
$\bar\psi(u)=u\cdot v(u)$ be a permutation of~$\mu_3$, which is
exactly~\ref{H2}.  By Theorem~\ref{thm:cpp-general}, $F$ is a
permutation polynomial of~$\Fq$.
\end{proof}

\begin{remark}\label{rem:mod9}
The hypothesis $q\equiv 1\pmod{9}$ is used only to ensure that
$3\mid s$, which in turn guarantees $\gcd(r,3)=1$.  If
$q\equiv 1\pmod{3}$ but $q\not\equiv 1\pmod{9}$, then $3\nmid s$, and
the congruence $r\equiv 1\pmod{s}$ no longer forces
$r\equiv 1\pmod{3}$.  The counterexamples in Section~\ref{sec:examples}
show that the construction can indeed fail in this range.
\end{remark}

\begin{corollary}\label{cor:family}
Under the hypotheses of Theorem~\textup{\ref{thm:scalar}}, suppose
that $v(u)=\tau(u)^s$ takes a constant value $\alpha\in\mu_3$ for all
$u\in\mu_3$.  Then conditions~\textup{\ref{H1}} and~\textup{\ref{H2}}
are satisfied, and $f(X)=X^r c(X^s)$ is a complete permutation
polynomial of\/~$\Fq$.
\end{corollary}

\begin{proof}
Since $\alpha=\tau(u)^s\ne 0$, the scalar $\tau(u)$ is nonzero for
each $u\in\mu_3$, which is~\ref{H1}.  The induced map is
$\bar\psi(u)=u\cdot\alpha=\alpha u$, a rotation of the cyclic
group~$\mu_3$, hence a bijection.  This verifies~\ref{H2}, and the
result follows from Theorem~\ref{thm:scalar}.
\end{proof}

\section{Examples and counterexamples}\label{sec:examples}

In this section we illustrate the family criterion
(Corollary~\ref{cor:family}) with explicit CPP constructions, and then
demonstrate through counterexamples that the congruence
$q\equiv 1\pmod{9}$ is essential for our method.

Throughout, we fix $\delta\in\Fq$ with $\delta^2+\delta+1=0$ and take
$c(u)=u^2+\delta u+\delta^2+1$, so that $\delta$ plays the role
of~$\omega$ and $\mu_3=\{1,\delta,\delta^2\}$.  From the proof of
Theorem~\ref{thm:bbd3}(1),
\[
  c(1)=1, \qquad c(\delta^2)=1, \qquad c(\delta)=3\delta^2+1,
\]
and $c(u)^s=1$ for all $u\in\mu_3$ whenever $(3\delta^2+1)^s=1$.  We
work in the scalar fiber case $r=1+ks$ and define
$f(X)=X^r c(X^s)$ and $F(X)=f(X)+X$.  For each $u\in\mu_3$, we write
$\tau(u)=1+c(u)\,u^k$ and $v(u)=\tau(u)^s$ as in
Theorem~\ref{thm:scalar}.

\subsection{CPP examples}
For each example, we verify that $\tau(u)\ne 0$ for all $u\in\mu_3$
and that $v(u)$ is constant on~$\mu_3$, so that
$\bar\psi(u)=u\,v(u)$ is a rotation of~$\mu_3$
(Corollary~\ref{cor:family}).

\medskip
\noindent\textbf{(a)\; $q=109$, $s=36$, $\delta=63$, $k=2$
($r=73$).}\;
One checks $63^2+63+1=4033=37\times 109\equiv 0$, confirming
$\delta\in\mu_3$, with $\delta^2\equiv 45\pmod{109}$.
We compute $c(\delta)=3\cdot 45+1\equiv 27\pmod{109}$
and verify $27^{36}\equiv 1\pmod{109}$, so
$c(u)^s=1$ for all $u\in\mu_3$.

Since $k=2$, the cubes of unity satisfy $1^2=1$, $\delta^2\equiv 45$,
$(\delta^2)^2=\delta^4=\delta\equiv 63$, so:
\begin{align*}
  \tau(1)      &= 1+c(1)\cdot 1^2        = 1+1     = 2,  \\
  \tau(\delta) &= 1+c(\delta)\cdot\delta^2 = 1+27\cdot 45
                \equiv 1+16 = 17 \pmod{109},\\
  \tau(\delta^2) &= 1+c(\delta^2)\cdot\delta = 1+63 = 64.
\end{align*}
All three values are nonzero, verifying~\ref{H1}.  A direct computation
gives $v(u)=\tau(u)^{36}\equiv 1\pmod{109}$ for each $u\in\mu_3$, so
$\bar\psi(u)=u\cdot 1=u$ is the identity on~$\mu_3$.  By
Corollary~\ref{cor:family}, $f(X)=X^{73}\,c(X^{36})$ is a complete
permutation polynomial of~$\mathbb{F}_{109}$.

\medskip
\noindent\textbf{(b)\; $q=163$, $s=54$, $\delta=58$, $k=3$
($r=163$).}\;
We have $58^2+58+1=3423=21\times 163\equiv 0$, with
$\delta^2\equiv 104\pmod{163}$.  Compute
$c(\delta)=3\cdot 104+1\equiv 150\pmod{163}$ and verify
$150^{54}\equiv 1\pmod{163}$.

Since $k=3$ and $u^3=1$ for all $u\in\mu_3$, we have
$\tau(u)=1+c(u)$ for each~$u$:
\begin{align*}
  \tau(1)        &= 1+c(1)=2, \\
  \tau(\delta)   &= 1+c(\delta)=1+150=151, \\
  \tau(\delta^2) &= 1+c(\delta^2)=2.
\end{align*}
All three values are nonzero.  A direct computation gives
$v(u)=\tau(u)^{54}\equiv 104\equiv\delta^2\pmod{163}$ for each
$u\in\mu_3$.  Since $v$ is constant with value $\delta^2\in\mu_3$,
the map $\bar\psi(u)=\delta^2 u$ is a nontrivial rotation of~$\mu_3$:
\[
  \bar\psi(1)=\delta^2, \qquad
  \bar\psi(\delta)=\delta^3=1, \qquad
  \bar\psi(\delta^2)=\delta^4=\delta.
\]
By Corollary~\ref{cor:family},
$f(X)=X^{163}\,c(X^{54})$ is a complete permutation polynomial
of~$\mathbb{F}_{163}$.

\medskip
\noindent\textbf{(c)\; $q=199$, $s=66$, $\delta=106$, $k=3$
($r=199$).}\;
We have $106^2+106+1=11343=57\times 199\equiv 0$, with
$\delta^2\equiv 92\pmod{199}$.  Compute
$c(\delta)=3\cdot 92+1\equiv 78\pmod{199}$ and verify
$78^{66}\equiv 1\pmod{199}$.

Again $k=3$ and $u^3=1$, so $\tau(u)=1+c(u)$:
\begin{align*}
  \tau(1)        &= 2, \\
  \tau(\delta)   &= 1+78=79, \\
  \tau(\delta^2) &= 2.
\end{align*}
All three values are nonzero.  A direct computation gives
$v(u)=\tau(u)^{66}\equiv 106\equiv\delta\pmod{199}$ for each
$u\in\mu_3$.  Since $v$ is constant with value $\delta\in\mu_3$,
the map $\bar\psi(u)=\delta u$ is a nontrivial rotation:
\[
  \bar\psi(1)=\delta, \qquad
  \bar\psi(\delta)=\delta^2, \qquad
  \bar\psi(\delta^2)=\delta^3=1.
\]
By Corollary~\ref{cor:family},
$f(X)=X^{199}\,c(X^{66})$ is a complete permutation polynomial
of~$\mathbb{F}_{199}$.

\subsection{Counterexamples outside $q\equiv 1\pmod{9}$}
We show that the construction can fail when $q\equiv 1\pmod{3}$ but
$q\not\equiv 1\pmod{9}$, even when~$r$ satisfies all coprimality
conditions.  In each case we use the same structural shape
$F(X)=X^r(X^{2s}+\delta X^s+\delta^2+1)+X$ with
$\delta^2+\delta+1=0$.

\medskip
\noindent\textbf{Counterexample~1: $q=7$.}\;
Here $7\equiv 1\pmod{3}$ but $7\not\equiv 1\pmod{9}$.  We have $s=2$
and $\delta=2$ (since $4+2+1=7\equiv 0$), giving
$\delta^2+1\equiv 5\pmod{7}$.  Taking $r=1$ (which satisfies
$\gcd(1,6)=1$):
\[
  F(X) = X(X^4+2X^2+5)+X = X^5+2X^3+6X = X(X^4+2X^2+6).
\]
For $x\in\{3,4\}$, we have $x^2\equiv 2\pmod{7}$, so the inner
factor evaluates to $4+4+6=14\equiv 0\pmod{7}$.  Therefore
$F(0)=F(3)=F(4)=0$: three distinct elements map to~$0$,
and $F$ is not a permutation of~$\mathbb{F}_7$.

\medskip
\noindent\textbf{Counterexample~2: $q=31$.}\;
Here $31\equiv 4\pmod{9}$, so $31\not\equiv 1\pmod{9}$.  We have $s=10$
and $\delta=25$ (since $625+25+1=651=21\times 31\equiv 0$), with
$\delta^2\equiv 5$ and $\delta^2+1\equiv 6\pmod{31}$.  Taking $r=7$,
which satisfies $\gcd(7,30)=1$, $\gcd(7,10)=1$, and $\gcd(7,3)=1$:
\[
  F(X) = X^7(X^{20}+25X^{10}+6)+X.
\]
Since $2^5=32\equiv 1\pmod{31}$ and $8=2^3$, we have
$8^5=2^{15}=(2^5)^3\equiv 1$, hence $8^{10}\equiv 1$ and
$8^{20}\equiv 1$.  The inner factor at $x=8$ evaluates to
$1+25+6=32\equiv 1\pmod{31}$, so
\[
  F(8) = 8^7\cdot 1+8.
\]
Now $8^7=8^5\cdot 8^2\equiv 1\cdot 64\equiv 2\pmod{31}$, giving
$F(8)\equiv 2+8=10\pmod{31}$.

For $x=5$: we compute $5^2=25$, $5^4\equiv 625\equiv 5\pmod{31}$,
so $5^{10}=(5^5)^2=(5^4\cdot 5)^2=25^2\equiv 5$ and
$5^{20}=(5^{10})^2=25\pmod{31}$.  The inner factor evaluates to
$25+25\cdot 5+6=156\equiv 1\pmod{31}$, and
$5^7=5^4\cdot 5^2\cdot 5=5\cdot 25\cdot 5=625\equiv 5\pmod{31}$.
Therefore $F(5)=5\cdot 1+5=10$.

Since $F(5)\equiv F(8)\equiv 10$ with $5\ne 8$, the map~$F$ is not
injective.

\medskip
These counterexamples do not claim that no CPPs exist outside
$q\equiv 1\pmod{9}$; rather, they show that the natural parameter
choices used in our constructions can fail in that range, and the
congruence $q\equiv 1\pmod{9}$ is genuinely used by our criteria.

\bigskip

\noindent\textsc{Chahrazade Bouyacoub},
Department of Mathematics, LA3C Laboratory
\par\noindent\textit{e-mail}: \texttt{chahrazadelamia.bouloudene@usthb.edu.dz}

\medskip

\noindent\textsc{Asmae El-Baz},
Département de mathématiques et de statistique,
Université Laval, Québec, QC G1V 0A6, Canada
\par\noindent\textit{e-mail}: \texttt{asmae.el-baz.1@ulaval.ca}

\medskip

\noindent\textsc{Omar Kihel},
Department of Mathematics, Brock University,
Ontario, Canada L2S 3A1
\par\noindent\textit{e-mail}: \texttt{okihel@brocku.ca}
\end{document}